\documentclass[12pt]{amsart}
\usepackage{a4}
\usepackage{amsthm, amsfonts, amssymb,latexsym}
\usepackage{enumerate, color}
\usepackage[english]{babel}%[english]
\usepackage[latin1]{inputenc}

\newtheorem{lemma}{Lemma}[section]

\newtheorem{theorem}{Theorem}[section]

\newtheorem{cor}{Corollary}[section]
\numberwithin{equation}{section}

\newcommand{\beq}[1]{\begin{equation}\label{#1}}
\newcommand{\eeq}{\end{equation}}

\title[Products of sum sets]{A short proof of a near-optimal cardinality estimate for the product of a sum set}

\author[ O. Roche-Newton]{Oliver Roche-Newton}

\address{O. Roche-Newton: Johann Radon Institute for Computational and Applied Mathematics (RICAM), Austrian Academy of Sciences, 4040 Linz, Austria }
\email{o.rochenewton@gmail.com }

\begin{document}

\begin{abstract}
In this note it is established that, for any finite set $A$ of real numbers, there exist two elements $a,b \in A$ such that

$$|(a+A)(b+A)| \gg \frac{|A|^2}{\log |A|}.$$

In particular, it follows that $|(A+A)(A+A)| \gg \frac{|A|^2}{\log |A|}$. The latter inequality had in fact already been established in an earlier work of the author and Rudnev \cite{rectangles}, which built upon the recent developments of Guth and Katz \cite{GK} in their work on the Erd\H{o}s distinct distance problem. Here, we do not use those relatively deep methods, and instead we need just a single application of the Szemer\'{e}di-Trotter Theorem. The result is also qualitatively stronger than the corresponding sum-product estimate from \cite{rectangles}, since the set $(a+A)(b+A)$ is defined by only two variables, rather than four. One can view this as a solution for the pinned distance problem, under an alternative notion of distance, in the special case when the point set is a direct product $A \times A$. Another advantage of this more elementary approach is that these results can now be extended for the first time to the case when $A \subset \mathbb C$.
\end{abstract} 

\maketitle
\section{Introduction}
In this note, we consider a variation on the sum-product problem, in which the aim is to show that certain sets defined by a combination of additive and multiplicative operations will always be large. For example, given a finite set $A$ of real numbers, define
$$(A-A)(A-A):=\{(a-b)(c-d):a,b,c,d \in A\}.$$
By the same heuristic arguments that support the Erd\H{o}s-Szemer\'{e}di sum-product conjecture, one expects that $(A-A)(A-A)$ will always be large in comparison to the input set $A$. In \cite{rectangles}, the following\footnote{Here and throughout this paper, for positive values $X$ and $Y$ the notation $X \gg Y$ is used as a shorthand for $X\geq cY$, for some absolute constant $c>0$. If both $X \gg Y$ and $X \ll Y$ hold, we may write $X \approx Y$.} bound was established which showed that this is indeed the case:
\begin{equation}
|(A-A)(A-A)| \gg \frac{|A|^2}{\log |A|}.
\label{sp1}
\end{equation}
The same argument in \cite{rectangles} yields the same lower bound for $|(A+A)(A+A)|$. Some other interesting results in this direction can be found in \cite{ENR}, \cite{TJthesis}, \cite{TJ}, \cite{MORNS}, \cite{RSS} and \cite{ungar}, amongst others.

In all of the aforementioned works, incidence geometry plays a central role. An extremely influential result in this area is the Szemer\'{e}di-Trotter Theorem, which says that, given finite sets $P$ and $L$ of points and lines respectively in $\mathbb R^2$, the number of incidences between $P$ and $L$ satisfies the upper bound
\begin{equation}
|\{(p,l) \in P \times L : p \in l\}| \ll |P|^{2/3}|L|^{2/3}+|P|+|L|.
\label{ST}
\end{equation}
The quantity on the left hand side of the above inequality is usually denoted by $I(P,L)$. Incidence geometry also played a central role in the recent landmark work of Guth and Katz \cite{GK} on the Erd\H{o}s distinct distances problem. Guth and Katz established an incidence bound for points and lines in $\mathbb R^3$, which was then used to prove that for any finite set $P$ of points in $\mathbb R^2$, the set of distinct distances determined by $P$ has near-linear size. To be precise, they proved that
\begin{equation}
|\{d(p,q):p,q \in P\}| \gg \frac{|P|}{\log |P|},
\label{GK}
\end{equation}
where $d(p,q)$ denotes the Euclidean distance between $p$ and $q$. Note that the example $P=[N] \times [N]$, where $[N]=\{1,2,\dots,N\}$, illustrates that this bound is close to best possible.

One of the tools that Guth and Katz use in their analysis is the Szemer\'{e}di-Trotter Theorem. They also introduced polynomial partitioning, and utilise some non-trivial facts from algebraic geometry.

In \cite{rectangles} the authors considered the pseudo-distance $R(p,q)$ in place of $d(p,q)$, where $R(p,q)$ denotes the (signed) area of the axis-parallel rectangle with $p$ and $q$ at opposite corners. To be precise, for two points $p=(p_1,p_2)$ and $q=(q_1,q_2)$ in the plane, we define
$$R(p,q):=(p_1-q_1)(p_2-q_2)$$ 
It was then possible to apply the incidence result of Guth and Katz to establish that
\begin{equation}
|\{R(p,q):p,q \in P\}| \gg \frac{|P|}{\log |P|},
\label{recdis}
\end{equation}
and \eqref{sp1} followed as a corollary after taking $P= A \times A$. Once again, the example $P=[N] \times [N]$ shows that this bound is close to best possible.

In this note, we prove the following result which strengthens \eqref{sp1}:

\begin{theorem} \label{mainthm} For any set $A \subset \mathbb R$, there exist elements $a,a' \in A$ such that
$$|(A-a)(A-a')| \gg \frac{|A|^2}{\log |A|}.$$
\end{theorem}

Here, we obtain quadratic growth for a set which depends on only two variables. There are similarities here with the Erd\H{o}s pinned distance problem, where the aim is to show that, for any finite set $P\subset \mathbb R^2$, there exists $p \in P$ such that
$$|\{d(p,q):q \in P\}| \gg \frac{|P|}{\sqrt{\log |P|}}.$$
This harder version of the Erd\H{o}s distinct distance problem remains open, with the current best-known result, due to Katz and Tardos \cite{KT}, stating that there exists $p \in P$ such that 
$$|\{d(p,q):q \in P\}|\gg |P|^{\alpha},$$ 
where $\alpha \approx 0.864$. However, Theorem \ref{mainthm} shows that, if we instead consider the pseudo-distance $R(p,q)$ then we have a near-optimal bound for the corresponding pinned distance problem, in the special case when $P=A \times A$ is a direct product. Such a result, even with the additional direct product restriction, is not currently known for Euclidean distance.

Another advantage of the approach in this paper is that the proof is relatively straightforward. In particular, we obtain a new proof of \eqref{sp1}, and in fact a stronger result, without utilising the Guth-Katz machinery.

This paper is closely related to work contained in the PhD thesis of Jones \cite{TJthesis} on the growth of sets of real numbers. In fact, the main lemma here, the forthcoming Lemma \ref{mainlemma}, forms part of the proof of \cite[Theorem 5.2]{TJthesis}, although it is expressed rather differently there in terms of the notion of the cross-ratio. Consequently, we are able to give a new proof of Theorem 5.2 from \cite{TJthesis}; that is we establish the following three-variable expander bound
$$\left|\left\{\frac{a-b}{a-c}:a,b,c \in A\right\}\right| \gg \frac{|A|^2}{\log |A|}.$$

It appears that the proof here is more straightforward than the one originally given by Jones \cite{TJthesis}.

The only major tool needed in this paper is the Szemer\'{e}di-Trotter Theorem. In particular, we use the following standard corollary of \eqref{ST} which bounds the number of rich lines in an incidence configuration:

\begin{cor}\label{STcor}[Szemer\'{e}di-Trotter Theorem]

Let $P$ be a set of points in $\mathbb R^2$ and let $k \geq 2$ be a real number. Define $L_k$ to be the set of lines containing at least $k$ points from $P$. Then
\begin{equation}
|L_k| \ll \frac{|P|^2}{k^3}+\frac{|P|}{k}.
\label{STcor1}
\end{equation}
In particular, if $k \leq |P|^{1/2}$, then
\begin{equation}
|L_k| \ll \frac{|P|^2}{k^3}.
\label{STcor2}
\end{equation}

\end{cor}

\section{Energy bound}
\begin{lemma} \label{mainlemma} Let $Q$ denote the number of solutions to the equation
\begin{equation} \label{energy}
(a-b)(a'-c')=(a-c)(a'-b')
\end{equation}
such that $a,a',b,b',c,c' \in A$. Then 
$$Q \ll |A|^4\log |A|.$$
\end{lemma}

\begin{proof}
First of all, the number solutions to \eqref{energy}  of the form
$$(a-b)(a'-c')=(a-c)(a'-b')=0,$$
is at most $4|A|^4$. Also, there are at most $|A|^4$ trivial solutions whereby $b=c$. Now, let $Q^*$ denote the number of solutions to
\begin{equation} \label{energy*}
(a-b)(a'-c')=(a-c)(a'-b')\neq 0,\,\,\,\,\,\,\,\,\,\,\,\,b \neq c.
\end{equation}
This is the same as the number of solutions to
\begin{equation} \label{energy**}
\frac{a-b}{a'-b'}=\frac{a-c}{a'-c'}\neq 0,\,\,\,\,\,\,\,\,\,\,\,\,\, b \neq c.
\end{equation}
Let $P=A\times A$ and let $L(P)$ denote the set of lines determined by $P$. That is, $L(P)$ is the set of lines supporting $2$ or more points from the set. Note that $a,a',b,b',c$ and $c'$ satisfy \eqref{energy**} only if the points $(a',a),(b',b)$ and $(c',c)$ from $P$ are collinear and distinct. Therefore,
\begin{align*}
Q^* &\leq \sum_{l\in L(P)}|l \cap P|^3
\\& \ll \sum_j \sum_{2^{j} \leq |l \cap P| < 2^{j+1}} |l \cap P|^3,
\end{align*}
where $j$ ranges over all positive integers such that $2^{j} \leq |A|$. Note that there are no lines in $L(P)$ which contain more than $|A|$ points from $P$, which is why this sum does not need to include any larger values of $j$.

For the aforementioned range of values for $j$, it follows from Corollary \ref{STcor}, and in particular bound \eqref{STcor2}, that
$$|\{l: |l\cap P| \geq 2^{j}\}| \ll \frac{|P|^2}{(2^j)^3}.$$
Therefore,
$$Q^* \ll \sum_j |P|^2 \ll |A|^4 \log|A|.$$
Finally, $Q \ll |A|^4 +Q^* \ll |A|^4 \log |A|$, as required.
\end{proof}

\begin{cor} \label{TJ}
For any finite set $A \subset \mathbb R$,
$$\left|\left\{\frac{a-b}{a-c}:a,b,c \in A\right\}\right| \gg \frac{|A|^2}{\log |A|}.$$
\end{cor}
\begin{proof}
Let 
$$n(x):=\left| \left\{(a,b,c) \in A^3:\frac{a-b}{a-c}=x\right\}\right|$$
denote the number of representations of $x$ as an element of the set in question. We know that
$$|A|^3 \ll |A|^3-|A|^2 = \sum_x n(x).$$
Also, the quantity $\sum_x n^2(x)$ is strictly\footnote{The quantity $\sum_x n^2(x)$ is the number of solutions to \eqref{energy}, minus the number of solutions for which $a=c$ or $a'=c'$.}  less than the number of solutions to \eqref{energy}. Therefore, it follows from the Cauchy-Schwarz inequality and Lemma \ref{mainlemma} that
\begin{align*}
|A|^6 \ll \left(\sum_x n(x)\right)^2 &\leq \left|\left\{\frac{a-b}{a-c}:a,b,c \in A\right\}\right|  \sum_x n^2(x)
\\& \ll  \left|\left\{\frac{a-b}{a-c}:a,b,c \in A\right\}\right|  |A|^4 \log |A|,
\end{align*}
and the result follows after rearranging this inequality.

\end{proof}
\subsection{Remarks}
Let $E^*(A,B)$ be the \textit{multiplicative energy} of $A$ and $B$; that is, the number of solutions to
$$ab=a'b'$$
such that $a,a' \in A$ and $b,b' \in B$. Using this notation, Lemma \ref{mainlemma} can be expressed in the form of the following bound:
\begin{equation}
\sum_{a,a' \in A}E^*(a-A,a'-A) \ll |A|^4 \log |A|.
\label{energysum}
\end{equation}
See \cite[Lemma 2.4]{MORNS} for a similar bound on the sum of multiplicative energies after different additive shifts.

The proof of Lemma \ref{mainlemma} can undergo a number of small modifications in order to deduce slightly different results involving multiple sets $A,B,C, \dots \in \mathbb R$ of approximately the same size. For example, if we instead take $P=(A \cup B) \times (A \cup B)$, where $|B| \approx |A|$, then the number of solutions to \eqref{energy**} such that $a,a' \in A $ and $b,b',c,c' \in B$ is less than the number of collinear triples in the point set $P$. After repeating the argument of Lemma \ref{mainlemma}, it follows that
\begin{equation}
\sum_{a,a' \in A}E^*(a-B,a'-B) \ll |A|^4 \log |A|.
\label{energysum2}
\end{equation}
In particular, if $B=-A$, this yields
\begin{equation}
\sum_{a,a' \in A}E^*(a+A,a'+A) \ll |A|^4 \log |A|.
\label{energysum3}
\end{equation}
\section{Proof of Theorem \ref{mainthm}}
It follows from \eqref{energysum} that there exist $a,a' \in A$ such that 
\begin{equation}
E^*(a-A,a'-A) \ll |A|^2 \log |A|.
\label{exists}
\end{equation}
We also have the following well-known bound for the multiplicative energy, which follows from an application of the Cauchy-Schwarz inequality:
\begin{equation}
E^*(A,B) \geq \frac {|A|^2|B|^2}{|AB|}.
\label{CS}
\end{equation}
After comparing \eqref{exists} and \eqref{CS}, it follows that
$$|(A-a)(A-a')| \gg \frac{|A|^2}{\log |A|},$$
as required. $\square$
\subsection{Remark} By the same argument, but utilising \eqref{energysum3} in place of \eqref{energysum}, it also follows that there exist $a,a' \in A$ such that
$$|(A+a)(A+a')| \gg \frac{|A|^2}{\log |A|}.$$

\section{The complex setting}

As stated in the abstract, an advantage of this more straightforward approach is that it allows for results that were previously only known for sets of real numbers to be extended to the complex setting. The only tool used in the proofs of Theorem \ref{mainthm} and Corollary \ref{TJ} is the Szemer\'{e}di-Trotter Theorem. It is now known that this theorem holds for sets of points and lines in $\mathbb C^2$ (this was first proven by \cite{toth}, with a more modern proof given by Zahl \cite{zahl}; see also Solymosi and Tao \cite{ST}).

One can therefore repeat the analysis of this paper verbatim in the complex setting, applying the complex Szemer\'{e}di-Trotter Theorem in place of the real version, and deduce exactly the same results for a set $A$ of complex numbers. 

In particular, we deduce that for any finite set $A \subset \mathbb C$, there exist $a,a' \in A$ such that
\begin{equation}
|(A-a)(A-a')| \gg \frac{|A|^2}{\log |A|}
\label{comp1}
\end{equation}
and it follows that
\begin{equation}
|(A-A)(A-A)| \gg \frac{|A|^2}{\log |A|}.
\label{comp2}
\end{equation}
Since the earlier proof of \eqref{comp2} for real $A$ in \cite{rectangles} was based on the three dimensional incidence bounds in \cite{GK}, it was not previously known that this bound extended to the complex setting. Similarly, the approach in this paper can be used to show that for any finite set $A \subset \mathbb C$, we have
$$\left|\left\{\frac{a-b}{a-c}:a,b,c \in A\right\}\right| \gg \frac{|A|^2}{\log |A|}.$$

\section*{Acknowledgements}The author was supported by the Austrian Science Fund (FWF): Project F5511-N26, which is part of the Special Research Program ``Quasi-Monte Carlo Methods: Theory and Applications". I am grateful to Brendan Murphy for his helpful feedback. I am also grateful to Ilya Shkredov for helpful conversations and especially for his interpretation of the work of Jones \cite{TJthesis}.

\end{document}